\documentclass[letterpaper, 10 pt, conference]{ieeeconf}

\IEEEoverridecommandlockouts 
\usepackage{graphicx} 
\usepackage{cite}
\usepackage{amsmath,amssymb,amsfonts}
\usepackage{url}
\usepackage{color, soul}
\usepackage{hyperref}
\usepackage{graphicx}
\usepackage{caption}
\usepackage{subcaption}
\usepackage{soul}
\usepackage{optidef}
\usepackage{listings}
\usepackage{ amstext}
\usepackage{verbatim}
\usepackage{indentfirst}
\usepackage[utf8]{inputenc}
\usepackage{caption}
\usepackage{subcaption}
\usepackage{algorithm}
\usepackage{framed}
\usepackage{rotating}
\usepackage{cite}
\usepackage{optidef}
\usepackage{graphicx}
\usepackage{booktabs}
\usepackage{algpseudocode}
\usepackage{cleveref}
\algrenewcommand\algorithmicrequire{\textbf{Data:}}
\algrenewcommand\algorithmicensure{\textbf{Return:}}

\newcommand{\real}{\mathbb{R}}

\newcommand\gbf{\mathbf{g}}

\newcommand\pbf{\mathbf{p}}

\newcommand\vbf{\mathbf{v}}

\newcommand\xbf{\mathbf{x}}
\newcommand\ybf{\mathbf{y}}
\newcommand\zbf{\mathbf{z}}

\newtheorem{theorem}{Theorem}
\newtheorem{lemma}{Lemma}
\newtheorem{definition}{Definition}

\newtheorem{proposition}{Proposition}

\newtheorem{assumption}{Assumption}

\graphicspath{{/IMG/}} 

\begin{document}
\title{\LARGE \bf Induced Stackelberg Equilibrium Seeking via Iterative Tikhonov Regularization}

\author{Silvia Cianchi, Anibal Sanjab, Sergio Grammatico \thanks{ Silvia Cianchi is with Delft University of Technology, and VITO and EnergyVille  \href{mailto:s.cianchi@tudelft.nl}{\texttt{s.cianchi@tudelft.nl}}. Anibal Sanjab is with VITO and EnergyVille \href{mailto:anibal.sanjab@vito.be}{\texttt{anibal.sanjab@vito.be}}. Sergio Grammatico is with Delft University of Technology \href{mailto:s.grammatico@tudelft.nl}{\texttt{s.grammatico@tudelft.nl}}. This work was partially funded by the ERC under project ARGON.}}
\maketitle

\begin{abstract}
Existing methods for learning Stackelberg equilibria typically assume that the followers’ (variational, generalized) Nash equilibrium is unique. However, in the presence of multiple equilibria, without a selection convention, the problem may become ill-posed, thus leading standard algorithms to potentially fail to converge. This paper addresses this issue by introducing an optimal selection at the lower-level game, hereby defining a Stackelberg game with induced equilibrium selection. To this end, we enable the leader to augment the followers’ game with an additional vanishing term that acts as an incentive. We then propose a follower-agnostic zeroth-order method, whereby the leader converges to a solution of the resulting problem by iteratively probing the followers and jointly updating its decision variable and the incentive term.

\end{abstract}
\section{Introduction} 
Hierarchical decision-making models, in which a leader commits to a strategic decision anticipating the responses of a set of followers, arise in numerous engineering systems, including smart grids \cite{fochesato2022stackelberg, yu2015real}, energy-mobility systems \cite{maljkovic2023hierarchical}, wireless sensor networks \cite{nie2018stackelberg}, and cyber-physical systems \cite{xing2024denial, 7446354, 9144263}. These interactions can be modeled as Stackelberg games, where the leader minimizes an objective function subject to the followers playing a (generalized) Nash equilibrium of the lower-level game \cite{von2010market}, whose standard solution concept is the Stackelberg equilibrium (SE). 

In recent years, a growing body of work has focused on the efficient computation of SE, both in centralized full-information settings \cite{dempe2002foundations} and in interactive, partial-information frameworks \cite{ maheshwari2024follower, maljkovic2025decentralized, grontas2024big}. 
These approaches typically rely on the assumption that the lower-level equilibrium is unique, which is restrictive in many practical scenarios \cite{jeroslow1985polynomial}, and represents a fundamental theoretical limitation. In fact, when multiple equilibria exist, the follower response becomes inherently ambiguous, possibly rendering the notion of SE no longer well-defined, without a proper selection convention. 

A common approach to address this issue is to adopt pessimistic (or optimistic) SE formulations \cite{zemkoho2016solving,liu2020generic, dempe2014solution}, where the leader considers the worst-case (or best-case) scenario, and assumes that followers select the least (or the most) favorable equilibrium for its objective function. These formulations suffer of two main drawbacks: ($i$) they may be overly conservative (or overly optimistic); and ($ii$) they are primarily suitable for offline, centralized settings, where the leader knows the followers' responses. In contrast, in online and decentralized settings, where the leader iteratively probes the followers to converge to a solution, these approaches are inadequate, as the leader’s update depends on the equilibrium actually selected by the followers. An alternative approach is proposed in \cite{wang2022coordinating}, where an interactive, stochastic method is developed that is robust against arbitrary equilibrium selection. However, this solution relies on the assumption that the follower responses are drawn from an unbiased equilibrium flow, which may be difficult to verify in practice.

In the context of security games, the notion of inducible SE has been proposed \cite{guo2019inducibility}, based on the idea that the leader can enforce a favorable equilibrium selection at the lower-level through infinitesimal perturbations of its strategy. However, this mechanism fundamentally relies on structural properties of the problem. In fact, if the lower-level game is monotone, the follower equilibrium set is typically non-singleton for all leader decisions, thus perturbations of the leader’s strategy alone cannot enforce a unique follower response.
On the other hand, in variational inequality and monotone operator literature, regularization techniques, such as Tikhonov regularization, have long been used to enforce uniqueness and select specific equilibria. By introducing a strongly monotone regularization term, the resulting problem admits a unique solution, which converges to a particular equilibrium of the original problem as the regularization vanishes  \cite{yamada2005hybrid, xu2003convergence, cegielski2014properties}. This idea has been successfully applied in generalized Nash equilibrium problems \cite{benenati2023semi, benenatix2023optimal}, where a distributed cooperative algorithm leveraging Tikhonov regularization enables convergence to optimally selected equilibria. To the best of the authors' knowledge, regularization-based selection mechanisms have not been systematically integrated into Stackelberg games, where enforcing a desired selection among multiple lower-level equilibria introduces additional challenges. In fact, differently from GNE settings, where followers collaboratively converge to the optimal equilibrium, the leader here must actively enforce the selection. Moreover, the lower-level game depends on the leader decision, so the selection mechanism must account for both the followers’ equilibria and the leader’s objective.

\smallskip
\noindent
\emph{Paper contributions.}
The contributions of this paper are twofold. First, we analyze Stackelberg games in the presence of multiple follower reactions and introduce the principle of optimal selection to address the resulting ambiguity in the formulation and solution concept. This leads to the reformulation of the original problem as a \emph{Stackelberg game with induced optimal selection}, which is well-posed, and whose solutions are well-characterized. Second, we propose a follower-agnostic approach that enables the leader to solve the game without requiring explicit knowledge of the lower-level game, relying solely on iterative queries to the followers. Specifically, we let the leader augment the followers’ objective functions with a vanishing incentive term that is updated jointly with its decision variable. Building on tools from Tikhonov regularization, we show that this additional term enforces optimal selection throughout the iterations and enables convergence to a solution of the original problem. Numerical simulations considering an energy community, with the community manager acting as the leader, and the community members engaging in peer-to-peer (P2P) energy trading as the followers of a defined Stackelberg problem governing their strategic interactions, corroborate the theoretical findings.

\smallskip
\noindent
\emph{Paper organization.}
The paper is organized as follows: Section \ref{Formulation} formulates the problem, while an illustrative example is presented in Section\ref{Illustrative}. Section \ref{Contribution} presents the proposed methodology and the solution algorithm. Section \ref{NumRes} presents the simulation results, and Section \ref{Conclusion} concludes the paper.

\smallskip
\noindent
\emph{Notation.}
Let $\xbf=[x_1, ... , x_n] \in \real^n$ denote a vector, and $x_t$ be its $t$-th component. The operator $\mathrm{col}(b_1,\dots,b_n)$ stacks the elements $b_1,\dots,b_n$ in a column vector. $\nabla_1 f(\zbf, \xbf), \nabla_2f(\zbf, \xbf),$ denote the gradients of $f$ with respect to the first and second variable, respectively. $\mathrm{VI}(F,\Omega)$ defines a variational inequality (VI) problem, whose solution set is $\mathrm{SOL}(F,\Omega)=\left\{\xbf^* \in \Omega\,|\,F(\xbf^*)^\top(\xbf-\xbf^*)\ge0, \forall \xbf \in \Omega)\right\}$. Finally, $\mathrm{N}_\mathcal{X}(\xbf)$ expresses the normal cone of $\mathcal{X}\subseteq \real^n$ at $\xbf \in \mathcal{X}$, i.e., $\mathrm{N}_\mathcal{X}(\xbf):=\left\{\vbf \in \real^n | \langle \vbf, \zbf-\xbf\rangle \le 0 \, \forall \zbf \in \mathcal{X}\right\}$.

\section{Problem Formulation}\label{Formulation}
We consider a one-leader (upper level) multiple-follower (lower level) Stackelberg game.
\subsection{Lower level}
Consider $N$ followers. Let $\xbf_i \in \real^{n_i}$ be the strategy of follower $i \in \mathcal{I}= \left\{1, \dots, N\right\}$, $\xbf=\mathrm{col}\left ((\xbf_i)_{i \in \mathcal{I}}\right)$, of size $n=\sum_{i \in \mathcal{I}}n_i$, be the collection of the strategies of all the followers, and $\xbf_{-i}$ be the collection of the strategies of all the followers except follower $i$. For a fixed value of the leader's strategy, $\ybf \in \real^m$, the followers' game is defined as 
\begin{align}\label{Followers}
\forall \,i \in \mathcal{I}: \underset{\xbf_i \in \mathcal{X}_i(\xbf_{-i})} {\min} \, J_i(\xbf_i, \xbf_{-i}; \ybf),
\end{align}
where $J_i:\real^{n}\times\real^m \rightarrow \real$ is the objective function of follower $i$, and $\mathcal{X}_i(\xbf_{-i}) \subseteq\real^{n_i}$ its feasible set, which includes local and shared constraints. We introduce the following assumptions. 
\begin{assumption}\label{Ass1}
    For each $i\in \mathcal{I}$, the objective function $J_i(\, \cdot \,,\xbf_{-i};\ybf)$ in \eqref{Followers} is continuously differentiable and convex.
\end{assumption}
\begin{assumption}\label{Ass2}
    For each $i\in\mathcal{I}$, the feasible set $\mathcal{X}_i(\xbf_{-i})$ in \eqref{Followers} is nonempty, closed, and convex. 
\end{assumption}
As a solution concept for the follower game in \eqref{Followers}, we adopt the notion of generalized Nash equilibrium (GNE):
\begin{definition}{\textbf{Generalized Nash equilibrium:}}
   For a given $\ybf$, a solution $\xbf^*$ of the game in \eqref{Followers} is a GNE if for each $i \in \mathcal{I}$,

\begin{align}
      \xbf^*_i \in \quad &\underset{\xbf_i\in \mathcal{X}_i(\xbf^*_{-i})} {\arg \min} \, J_i(\xbf_i, , \xbf^*_{-i}; \ybf).
\end{align}
\end{definition}
In particular, we investigate the set of \emph{variational} generalized Nash equilibria ($v$-GNE), a subset of GNEs characterized by a well-established interpretation as the solution set of a VI \cite{facchinei2003finite}:
\begin{definition}{\textbf{Variational generalized Nash equilibrium:}}\label{vGNE}
For a given $\ybf$, a solution $\xbf^*$ of the game in \eqref{Followers} is a $v$-GNE if it is a solution of the following VI: $\mathrm{VI}(F(\cdot\,; \ybf), \Omega)$, where $F(\cdot\,; \ybf)$ denotes the pseudogradient operator, i.e., $F(\xbf; \ybf):=\mathrm{col}(\nabla_{\xbf_i} J_i(\xbf_i, \xbf_{-i};\ybf)_{i \in \mathcal{I}})$, and $\Omega=\prod_{i \in \mathcal{I}}\mathcal{X}_i(\xbf_{-i})$.
\end{definition}
Let us denote the set of $v$-GNEs of the follower game, given $\ybf$, as \begin{align}\label{MathcalEy}
\mathcal{E}(\ybf):=\mathrm{SOL}(F(\cdot;\ybf), \Omega).
\end{align}

Lemma \ref{ORA} states a well-established result from variational analysis \cite[Theorem 2.4.8]{facchinei2003finite}. 
\begin{lemma}\label{ORA}
    Let Assumption \ref{Ass1}, \ref{Ass2} hold. If the pseudogradient operator $F(\cdot\,;\ybf)$ is monotone on $\Omega$, then $\mathcal{E}(\ybf)$ in \eqref{MathcalEy} is a convex set. Moreover, if $F(\cdot;\ybf)$ is strongly monotone on $\Omega$, then $\mathcal{E}(\ybf)$ is a singleton (i.e., the $v$-GNE solution is unique). 
\end{lemma}
Motivated by Lemma \ref{ORA}, we introduce the usual assumption on the monotonicity of the follower game to guarantee that the set of $v$-GNEs is a convex set. 
\begin{assumption}\label{ConvexSet}
    For each $\ybf$, the pseudogradient operator of the follower game in \eqref{Followers}, $F(\cdot\,;\ybf)$, is monotone.
\end{assumption}
\subsection{Upper level}
The strategy of the leader is denoted by $\ybf \in \mathcal{Y}=\real^m$ and it is chosen with the aim of minimizing the cost $J_0: \real^m \times \real^n\rightarrow \real$ subject to the followers playing a $v$-GNE. The Stackelberg game is thus formalized as follows: 
\begin{equation}\label{Leader}
\begin{aligned}
\underset{\ybf \in \mathcal{Y}} {\min} \, &J_0(\ybf, \xbf)
\\
\mathrm{s.t.} \,& \, \xbf \in \mathcal{E}(\ybf). 
\end{aligned}
\end{equation}
Let us introduce the following assumption. 
\begin{assumption}\label{Ass4}
    For any given $\xbf$, the function $J_0(\cdot,\xbf)$ is uniformly $L_1$-Lipschitz continuous, and, for any given $\ybf$, the function $J_0(\ybf,\cdot)$ is uniformly $L_2$-Lipschitz continuous.
\end{assumption}
If the follower equilibrium set $\mathcal{E}(\ybf)$ is a singleton, then, for each $\ybf$, the lower-level response is uniquely determined, and denoted as:
\begin{align*}
    \xbf^*(\ybf):=\left\{\mathcal{E}(\ybf)\right\}.
\end{align*}
Therefore, in this case, the Stackelberg game in \eqref{Leader} reduces to the bilevel formulation 
\begin{align}\label{Bilevel}
\underset{\ybf \in \mathcal{Y}} {\min} \, &J_0(\ybf, \xbf^*(\ybf)), 
\end{align}
whose global minima are the so-called Stackelberg equilibria of \eqref{Leader} \cite{dempe2002foundations}. However, \eqref{Bilevel} is in general nonconvex. Thus, the goal in Stackelberg games is typically to find a stationary point of \eqref{Bilevel} \cite{maheshwari2024follower, grontas2024big}.

If, instead, the lower level problem in \eqref{Followers} admits multiple $v$-GNEs, i.e., the set $\mathcal{E}(\ybf)$ contains multiple solutions, then the follower response becomes set-valued, and the SE solution concept has to be further specified. In this case, we propose the introduction of an equilibrium selection mechanism, a strongly convex objective function $\phi(\xbf)$ \cite{benenatix2023optimal}. 
\begin{assumption}\label{StronglyConvex}
    The selection function $\phi$ is continuous and $\mu$-strongly convex.
\end{assumption}

Under Assumptions [\ref{Ass1}--\ref{ConvexSet}, \ref{StronglyConvex}], the reaction of the followers to a given $\ybf$ becomes single-valued:
\begin{align}\label{OptSelection}
    \xbf^*_\phi(\ybf):=\arg \min \limits_{ \xbf \in \mathcal{E}(\ybf)} \phi(\xbf).
\end{align}

As such, the original game in \eqref{Leader} can be rewritten as a \emph{Stackelberg game with equilibrium selection},
which, differently from \eqref{Leader}, always admits a well-posed bilevel characterization:
\begin{align}\label{StazPhi}
        \min \limits_{\ybf\in \mathcal{Y}} J_0(\ybf, \xbf^*_\phi(\ybf)).
    \end{align}
In view of \eqref{StazPhi}, let us introduce the notion of $\phi$-induced Stackelberg equilibrium.
\begin{definition}{\textbf{Induced Stackelberg equilibrium:}}
    A solution $\left(\ybf^*, \xbf^*_\phi(\ybf^*)\right)$ of the game in \eqref{Leader} is a $\phi$-induced Stackelberg equilibrium if
    \begin{align}
        \ybf^* \in \arg\min \limits_{\ybf \in \mathcal{Y}} J_0(\ybf, \xbf^*_\phi(\ybf)),
    \end{align}
    with $\xbf^*_\phi(\ybf)$ as in \eqref{OptSelection}.
\end{definition}
We emphasize that the introduction of our selection mechanism generates an interpretable equilibrium characterization. 
We note that the bilevel objective optimization in \eqref{StazPhi} is typically nonconvex, so convergence to a globally optimal solution cannot be guaranteed, in general. Therefore,  similarly to what is done in the standard Stackelberg framework, our goal is to find stationary points of \eqref{StazPhi}.
Moreover, in line with \cite{maheshwari2024follower}, we consider the following additional assumption.
\begin{assumption}\label{AssStrong}
    The bilevel optimization function $J_0(\cdot\,, \xbf^*_\phi(\cdot))$ is twice-continuously differentiable, $\tilde{L}$-Lipschitz continuous, and $\tilde{\ell}$-smooth.
\end{assumption}
Assumption \ref{AssStrong} is imposed primarily for analytical convenience, but it can be relaxed. Indeed, if $J_0(\cdot\,, x^*_\phi(\cdot))$ is $\tilde{L}$-Lipschitz but nonsmooth, it can be approximated arbitrarily well by a smooth function via smoothing techniques \cite{nesterov2005smooth}, thereby enabling the application of our proposed algorithm under minor modifications.
\section{Illustrative Example: equilibrium selection does matter}\label{Illustrative}
In this section, we provide an illustrative example to show that, in the presence of multiple equilibria at the lower-level, without an explicit equilibrium selection rule, the concept of SE is ambiguous, and standard gradient-based algorithms may fail. Let us consider a simple $1$-leader $2$-follower Stackelberg game, with follower decision variables $x_1, x_2 \in \real$, $\xbf=\mathrm{col}(x_1,x_2)$:
\begin{subequations}\label{Lower-ILL}
\begin{align}
\underset{ x_1 \in \mathcal{X}_1} {\min} \, & J_1(x_1, x_2; y):=\frac{1}{2} (y+\epsilon)x_1^2+ x_1x_2\\
\underset{x_2 \in \mathcal{X}_2} {\min} \, & J_2(x_2, x_1; y):=\frac{1}{2} x_2^2+(y+\epsilon) x_1x_2,
\end{align}
\end{subequations}
with $\epsilon$ such that $(y+\epsilon)>0$, $\mathcal{X}_1=\left[x_1^{\min}, x_1^{\max}\right]$ and $\mathcal{X}_2=\left[x_2^{\min}, x_2^{\max}\right]$. 
The set of $v$-GNEs of \eqref{Lower-ILL}, $\mathcal{E}(y)$, is identified by solving the variational inclusion problem $\mathbf{0} \in F(\xbf; y)+N_{\mathcal{X}_1\times\mathcal{X}_2}(\xbf)$.

To avoid technicalities associated with boundary points, we consider cases where $\mathcal{X}_1$ and $\mathcal{X}_2$ are sufficiently large to be always inactive. Therefore, we focus exclusively on the set of interior $v$-GNEs which, for any fixed $y$, contains multiple solutions:
\begin{align*}
\mathcal{E}^\mathrm{int}(y)=\left\{\xbf\in \mathrm{int}(\mathcal{X}_1)\times \mathrm{int}(\mathcal{X}_2) \, |  \,x_2 = -(y+\epsilon)x_1\right\},
\end{align*}
whose projection onto the variable of the first follower is:
\begin{align*}
{\mathcal{E}}_1^\mathrm{int}(y)=\left\{x_{1}\, |\, \mathrm{col}\left(x_{1 }, \, -(y+\epsilon)x_{1 }\right) \in \mathrm{int}(\mathcal{X}_1) \times \mathrm{int}(\mathcal{X}_2)\right\}.
\end{align*}

At the upper level, the leader aims at solving:
\begin{equation}\label{Stackelberg-ILL}
\begin{aligned}
\underset{y} {\min} \, & J_0(y,\xbf):=y^2+y(x_1+x_2)
\\
\mathrm{s. t.} \, \,  &\xbf\in \mathcal{E}^\mathrm{int}(y).
\end{aligned}
\end{equation}

As explained in Section \ref{Formulation}, the follower reaction is set-valued. 

We now consider different scenarios of lower-level reactions to illustrate the impact of followers' equilibrium multiplicity on the performances of gradient-based algorithms for Stackelberg equilibrium seeking. 

In particular, we solve the game in \eqref{Stackelberg-ILL} via a first-order method, where, at each iteration, the leader observes the follower reaction, $\xbf_k$, that is, one solution picked from the set $\mathcal{E}^\mathrm{int}(y_k)$, and updates its decision variable with one step of gradient descent: 
\begin{align}\label{update}
    y_{k+1}=y_k-\eta \nabla_k,
\end{align}
where the vector $\nabla_k$ stands for
\begin{align}\label{ChainRule}
\nabla_k=\nabla_1 J_0(y_k, \xbf_k)+\nabla_2 J_0(y_k, \xbf_k) \frac{\mathrm{d} \xbf_k}{\text{d}y}.
\end{align}

We first consider the case where, at each iteration $k$, the variable of the first follower, $x_{1 \mid k}$, behaves as an exogenous variable, which takes values in $\mathcal{E}_1^{\mathrm{int}}(y_k)$ 
independently of the current decision variable of the leader, $y_k$. As such, the lower-level response is $\xbf_k=\mathrm{col}(x_{1 \mid k},\,-(y_k+\epsilon)\,x_{1 \mid k})$,
thus $\frac{\text{d}\xbf_k}{\text{d}y}=\mathrm{col}(0, \, -x_{1\mid k})$, and $\nabla_k$ in \eqref{ChainRule} takes the form:
\begin{align}\label{inexactNabla}
\nabla_k=2y_k(1-x_{1 \mid k})+x_{1 \mid k}(1-\epsilon).
\end{align}
Next, Proposition \ref{Cost/Osc} establishes a necessary condition on 
the sequence $\left(x_{1 \mid k}\right)_{k \in \mathbb{N}}$ ensuring the convergence of the iterates $\left(y_k\right)_{k \in \mathbb{N}}$ generated by \eqref{update} to a SE of the game \eqref{Stackelberg-ILL}. 
\begin{proposition}\label{Cost/Osc}
    Let us assume that $\left(y_k\right)_{k \in \mathbb{N}}$ generated by \eqref{update} converges, i.e., there exists $y_{\infty} \in \real$ such that $\lim \limits_{k \rightarrow \infty} y_k = y_{\infty}$. Then, there exists $ x_{1 \mid \infty}\in \real, \, x_{1 \mid \infty} \neq 1$, such that $\lim \limits_{k \rightarrow \infty} x_{1 \mid k} = x_{1 \mid\infty}$. Furthermore, $y_\infty$ is the SE solution of the game in \eqref{Stackelberg-ILL} with $\xbf=\mathrm{col}(x_{1 \mid\infty}, \, -(y+\epsilon)x_{1 \mid \infty})$.
\end{proposition}
\begin{proof}
    Consider the update rule in \eqref{update} with $\nabla_k$ as in \eqref{inexactNabla}. If $y_k \rightarrow y_{\infty} \in \real$, then necessarily $y_{k+1}-y_k \rightarrow 0$, which implies $2 y_{k}(1-x_{1 \mid k})+x_{1 \mid k}(1-\epsilon) \rightarrow 0$. Taking the limiting value, it holds that $\frac{- x_{1|k}(1-\epsilon)}{2(1-x_{1 \mid k})} \rightarrow y_\infty$. For this limit to exist and be finite, it necessarily holds $\lim \limits_{k \rightarrow \infty}\frac{x_{1 \mid k}}{1-x_{1 \mid k}}\in \real$, which implies $x_{1 \mid k} \rightarrow x_{1 \mid \infty}\in \real,$ for some $ x_{1 \mid \infty} \neq 1$.
\end{proof}

The plot in Fig. \ref{IllustrativeExample1} corroborates the above results, showing that if the sequence $\left(x_{1 \mid k}\right)_{k \in \mathbb{N}}$ persistently oscillates, violating the necessary condition in Proposition \ref{Cost/Osc}, then the iterates $\left(y_k\right)_{k \in \mathbb{N}}$, generated by \eqref{update}, do not converge.  

\begin{figure}
        \includegraphics[width=\columnwidth]{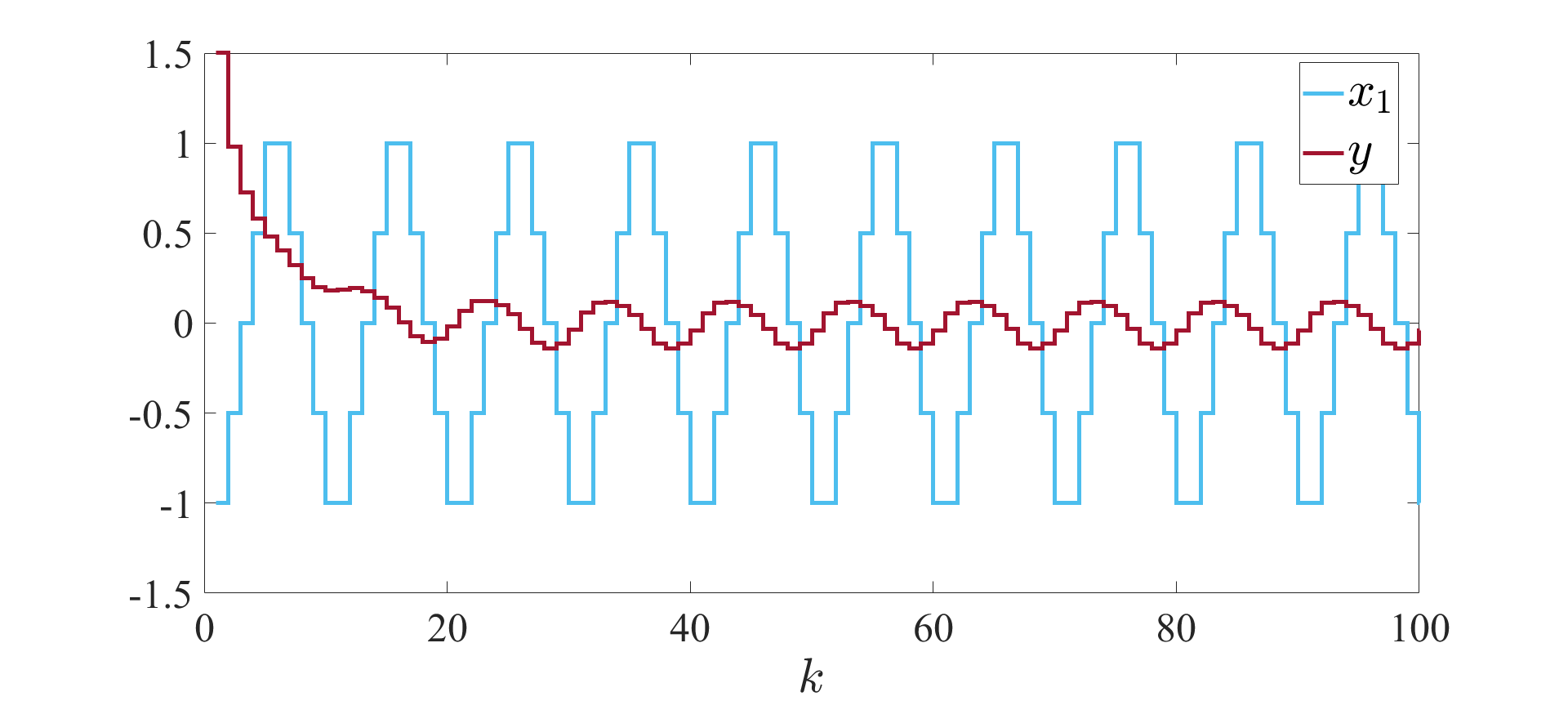}
        \captionsetup{width=\columnwidth}
        \caption{If the sequence $\left(x_{1 \mid k}\right)_{k \in \mathbb{N}}$ violates the necessary condition in Proposition \ref{Cost/Osc}, then the iterates $\left(y_k\right)_{k \in \mathbb{N}}$ generated by \eqref{update} fail to converge to a SE of the game \eqref{Stackelberg-ILL}.}
        \label{IllustrativeExample1}
\end{figure}

As a second scenario, we consider the case where a strongly monotone selection function $\phi(\xbf)=(x_1-1)^2+100 x_2^2$ is introduced, that enforces a single-valued follower reaction mapping, $\xbf^*_\phi(y) =\mathrm{col}\left( x^*_1(y), \, -(y+\epsilon)x_1^*(y)\right)$, with $x^*_1(y)=(1+100y^2)^{-1}\in {\mathcal{E}}_1^\mathrm{int}(y)$. As such, at every iteration $k$, it holds $\xbf_k=\xbf^*_\phi(y_k)$. We separately address two different information regimes: \begin{enumerate}
    \item \textit{Inexact regime:} the leader is not aware of the followers' selection mechanism, $\phi(\xbf)$, hence, it treats $x_{1\mid k}$ as a free variable, exogenously observed at each iteration $k$, and employs the update rule \eqref{update} with $\nabla_k$ in \eqref{inexactNabla} as before. Since the leader’s optimization is neglecting the dependence of $x^*_1(y)$ on $y$, this approach is not exact. The purple line in Fig. \ref{IllustrativeExample2} shows that the iterates generated by \eqref{update} converge to some $\bar{y}$, which is not a stationary point of the bilevel objective function $J_0(y, \xbf^*_\phi(y))$.
    \item \textit{Exact regime:} the leader is aware of the optimal selection, and computes the term $\nabla_k$ as in \eqref{ChainRule}, explicitly accounting for the induced mapping  $x^*_{1}(y)$. 
    As illustrated in the green line in Fig. \ref{IllustrativeExample2}), the iterates $\left(y_k\right)_{k \in \mathbb{N}}$ generated by \eqref{update} converge to the $\phi$-induced SE, $(y^*_\phi, \xbf^*_\phi(y^*_\phi))$, that is the global minimum of $J_0(y, \xbf_\phi^*(y))$. 
\end{enumerate}
\begin{figure}
        \includegraphics[width=\columnwidth]{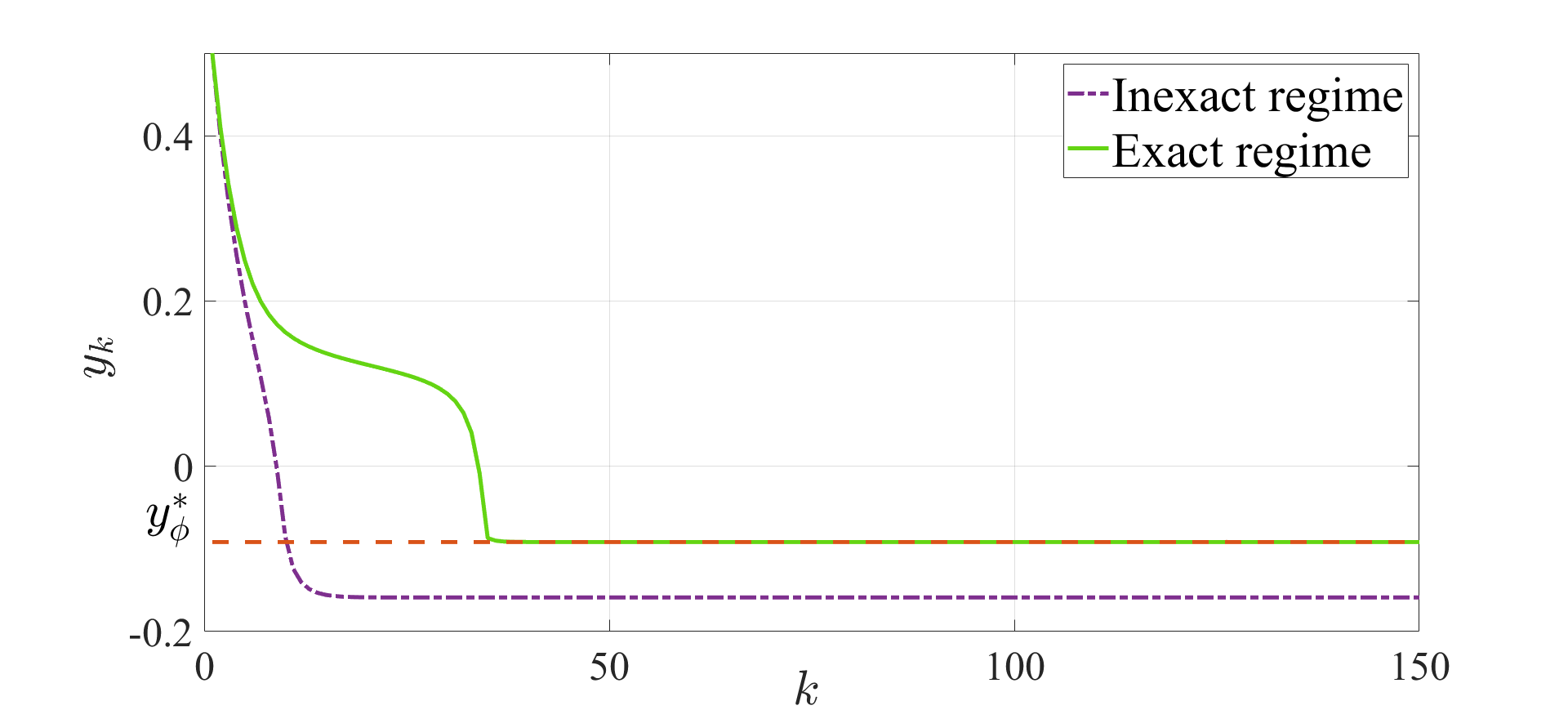}
        \captionsetup{width=\columnwidth}
        \caption{If a selection $\phi(\xbf)$ is introduced, the follower reaction, $\xbf^*_\phi(y)$, is uniquely determined by $y$. Without accounting for this dependence, the $\phi$-induced SE cannot be achieved.}
        \label{IllustrativeExample2}
\end{figure}

This illustrative example highlights two key points. First, in the presence of multiple equilibria at the lower-level, if an explicit equilibrium selection rule is not specified, the concept of SE is ambiguous, and standard algorithms may fail (e.g., oscillate). Second, when an optimal selection mechanism exists, it is beneficial for the leader to account for it, to achieve the corresponding $\phi$-induced SE, which is, by definition, optimal for the leader.
This motivates the adoption (or enforcement by the leader) of optimal equilibrium selection at the lower level, to provide a principled way to guarantee convergence to an interpretable solution.
\section{Main Contribution} \label{Contribution}
Motivated by the illustrative example in Section \ref{Illustrative}, we consider the Stackelberg game with equilibrium selection in \eqref{StazPhi}. In this section, we first propose a method which enables the leader to enforce the optimal selection mechanism at the lower-level, without requiring any knowledge about the followers' game. Then, we propose a zeroth-order algorithm that enables the leader to converge to a solution of the Stackelberg game with equilibrium selection through iterative interactions with the followers. 
\subsection{Tikhonov regularization theory}
We note that, for a fixed $\ybf$, the problem of optimal selection among $v$-GNEs in \eqref{OptSelection} is a hierarchical VI problem of the form \cite[Eq. (16)]{benenatix2023optimal}
\begin{equation}\label{VI}
\begin{aligned}
\underset{\xbf } {\min} \, &\phi(\xbf)
\\
\mathrm{s.t.} \,& \, \xbf \in \mathrm{SOL}(F(\cdot\,;\ybf), \Omega). 
\end{aligned}
\end{equation}

This enables the employment of tools from variational analysis, where the desired selection can be enforced by the use of Tikhonov-based regularization schemes \cite{facchinei2003finite}. The key idea of this approach lies in introducing a perturbation in the original problem; the solution of \eqref{VI} is proven to be the limit solution of the perturbed problem, as the perturbation tends to zero. Specifically, given $\beta >0$, the regularized problem is 
    \begin{align}\label{Tikh}
        \mathrm{VI}(F_\beta(\xbf;\ybf), \Omega),
    \end{align}
with regularized operator $F_\beta( \xbf; \ybf)=F(\xbf;\ybf)+\beta \nabla \phi(\xbf)$. 
In Lemma \ref{SERVEORA}, the main results of Tikhonov regularization theory are restated \cite[Theorem 12.2.5]{facchinei2003finite} for convenience.
\begin{lemma}\label{SERVEORA}
    For every $\beta >0$:
    \begin{itemize}
        \item The Tikhonov regularized problem in \eqref{Tikh} admits a unique solution $\xbf_\beta(\ybf)$;
        \item As $\beta \rightarrow 0$, the sequence $\left(\xbf_\beta(\ybf)\right)_{\beta\rightarrow 0}$ converges to the unique solution $\xbf^*_\phi(\ybf)$ of \eqref{OptSelection}. 
    \end{itemize}
\end{lemma}
The additional regularization term can be interpreted as an incentive mechanism introduced by the leader to steer the followers' reaction. In fact, the regularized operator $F_\beta(\xbf;\ybf)$ is the pseudogradient operator of a regularized game in which each follower’s objective function, $J_i(\xbf; \ybf)$, is augmented by the incentive term $\beta \phi(\xbf)$, introduced by the leader. The resulting regularized lower-level game is
\begin{align}\label{betaFollowers}
\forall \,i \in \mathcal{I}: \underset{\xbf_i \in \mathcal{X}_i(\xbf_{-i})} {\min} \, J_i(\xbf_i, \xbf_{-i}; \ybf)+\beta\phi(\xbf).
\end{align}

The Tikhonov regularization approach consists in perturbing the lower-level game with an additional incentive term, $\beta \phi(\xbf)$, to enforce the uniqueness of the follower reaction. Furthermore, by letting the regularization parameter, $\beta$, vanish, the unique regularized lower-level equilibrium, $\xbf_\beta(\ybf)$, converges to the optimal equilibrium of the original unperturbed lower-level game, $\xbf^*_\phi(\ybf)$. 
\subsection{Algorithm derivation}
We now discuss methods that enable the leader to solve the game with optimal selection in \eqref{StazPhi} in a follower-agnostic fashion. In view of the insights from Tikhonov regularization theory, a natural approach consists in solving, at every iteration, the full regularized lower-level problem in the limit $\beta \rightarrow 0$ to recover, for each given $\ybf_k$, the optimally selected lower-level equilibrium, $x^*_\phi(\ybf_k)$. In this case, since the follower response is unique, existing iterative methods to solve Stackelberg games, such as \cite{grontas2024big, maheshwari2024follower}, can be applied to update the leader’s decision. 
However, this nested scheme is computationally demanding, as it requires solving the full regularized lower-level problem to high accuracy at every iteration. Also, for vanishing values of $\beta$, the regularized follower problem becomes increasingly ill-conditioned. To overcome these limitations, we propose an alternative approach in which the regularization parameter $\beta_k$ can be updated jointly with the leader’s decision variable, $\ybf_k$, and be progressively driven to zero along the iterations, while still preserving the convergence to a stationary point of \eqref{StazPhi}.
 \begin{algorithm}
\caption{Zeroth-order induced SE seeking}\label{betaZOM}
\begin{algorithmic}[1]
   \Require Time horizon $K$, Initial conditions $\ybf_0 \in \real^m,$ Step sizes $(\eta_k)$, Perturbation radius $(\delta_k)$, Regularization parameters $(\beta_k)$
   \For{$k = 1$ to $K-1$}
   \State Sample $\vbf_k \sim \mathrm{Unif}(\mathcal{S}(\real^T))$;
   \State Assign $\hat{\ybf}_k=\ybf_k + \vbf_k \delta_k$;
   \State Collect $\xbf_k=\xbf_{\beta_k}(\ybf_k)$;
   \State Collect $\hat{\xbf}_k=\xbf_{\beta_k}(\hat{\ybf}_k)$;
   \State Compute $\hat{\gbf}_k=\frac{m}{\bar{\delta_k}}(J_0(\ybf_k, \xbf_k)- \,J_0(\hat{\ybf}_k, \hat{\xbf}_k))\vbf_k$;
   \State Update $\ybf_{k+1}=\ybf_k-\eta_k \hat{\gbf}_k$;
   \EndFor
\end{algorithmic}
\end{algorithm}
As such, we propose Algorithm \ref{betaZOM}, which is a zeroth-order method where, at each iteration, the leader sets the parameter $\beta_k$ of the incentive term $\beta_k \phi(\xbf)$, and queries the lower-level with the current value of its decision variable, $\ybf_k$, and with a perturbed version of it, $\hat{\ybf}_k=\ybf_k+\delta_k\vbf_k$, where $\delta_k \in \real$ is the perturbation radius and $\vbf_k \in \real^m$ a random direction vector. Based on the corresponding lower-level reactions, i.e., the solutions $\xbf_{\beta_k}(\ybf_k),$ and $ \xbf_{\beta_k}(\hat{\ybf}_k)$ of the regularized lower-level game \eqref{betaFollowers}, the leader computes the value of the proposed gradient estimator
\begin{align}\label{Gradient}
    \hat{\gbf}_k=\frac{m}{\delta_k}(J_0(\ybf_k,\xbf_{\beta_k}(\ybf_k))- \,J_0(\hat{\ybf}_k,\xbf_{\beta_k}(\hat{\ybf}_k))\vbf_k,
\end{align}
and updates its decision variable $\ybf_{k}$ with the following update rule
\begin{align}
    \ybf_{k+1}=\ybf_k-\eta_k\,\hat{\gbf}_k.
\end{align} 

The incentive term is also updated with a reduced value of the parameter $\beta_k$. 
Our proposed gradient estimator in \eqref{Gradient} builds on classical zeroth-order two-point gradient approximations \cite{nesterov2017random}. The difference with standard two-point gradient estimators \cite[Eq. 30]{nesterov2017random} is that, rather than evaluating the leader’s objective at the optimally selected follower equilibrium, i.e., 
\begin{align}\label{GradientVero}
    \gbf_k=\frac{m}{\delta_k}(J_0(\ybf_k,\xbf^*_\phi(\ybf_k))- \,J_0(\hat{\ybf}_k,\xbf^*_{\phi}(\hat{\ybf}_k))\vbf_k,
\end{align}
we use the solution of the regularized lower-level game at the current regularization level $\beta_k$, i.e., $\xbf_{\beta_k}(\ybf_k)$. This enables efficient gradient approximation without fully resolving the lower-level optimal selection problem at each iteration. For this reason, the choice of $\hat{\gbf}_k$ introduces an additional bias in the gradient estimator. However, by maintaining a suitable time-scale separation between $\beta_k$ and $\eta_k$, convergence is not hindered, as stated in our main result, Theorem \ref{BoundZOM}.
\begin{theorem}\label{BoundZOM}
Let $\eta_k=\bar{\eta}(k+1)^{-1/2}m^{-1}, \delta_k=\bar{\delta}(k+1)^{-1/4}m^{-1/2}$, such that $\bar{\eta}\le m/2 \tilde{\ell}$, and $\beta_k=\bar{\beta}(k+1)^{-\alpha}$ with $\alpha > \frac{1}{2}$. Then, the iterations $\left(\ybf_k\right)_{k \in [K]}$ of Algorithm \ref{betaZOM} are such that
    \begin{align}\label{Bound}
        \frac{1}{\sum_{k=1}^K \eta_k} \sum_{k=1}^K \eta_k \mathbb{E}\left[\left\|\nabla J_0(\ybf_k,\xbf^*_\phi(\ybf_k))\right\|^2\right]\le \mathcal{U},
    \end{align}
    where the expectation is taken with respect to $\vbf_k$, and $\mathcal{U}=\mathcal{U}_1+ \mathcal{U}_2+ \mathcal{U}_3+ \mathcal{U}_4$, with $\mathcal{U}_1=\frac{m (J_0(\ybf_0,\xbf^*_\phi(\ybf_0))-{J}_0^*)}{\bar{\eta}\sqrt{K}};\, \mathcal{U}_2=\frac{\tilde{\ell}^2m \bar{\delta}^2 \mathrm{log}(K)}{4 \sqrt{K}}; \, \mathcal{U}_3=\frac{m^3L_2^2 4 C\bar{\beta}^2}{\mu^2\bar{\delta}^2 \sqrt{K}};\,\mathcal{U}_4=\frac{4m\tilde{L}^2\tilde{\ell}\bar{\eta}\mathrm{log}(K)}{\sqrt{K}}$, assuming that $J_0(\ybf,\xbf^*_\phi(\ybf))\ge J^*_0 \,\,\forall \, \ybf \in \real^m$, and $C >0$.
\end{theorem}
\begin{proof}
   See the Appendix.
\end{proof}
\section{Numerical Results}\label{NumRes}
We consider a setup in which the lower-level agents engage in P2P electricity trading under network operational constraints, while the upper-level community manager sets the trading prices to minimize the overall community cost. For given trading prices, the lower-level game generally admits multiple equilibria, which can differ significantly in terms of community-level performance. In fact, some may be socially efficient and fair, while others may be inefficient or unfair, impacting metrics such as grid stress, energy self-sufficiency, internal trading volume, or user comfort. Consequently, the community manager has a strong incentive to guide the system, in a transparent way, toward an optimally selected equilibrium. For the lower-level problem, which is adapted from \cite{9732452}, and also used in \cite{benenatix2023optimal}, we consider the IEEE 13-bus distribution feeder, and we assume that each node is associated with one agent. In particular, each agent $i$ has decision authority on the power generated, $p^\mathrm{g}_{i,h}$, the power bought from the main grid, $p^\mathrm{mg}_{i,j}$, the power drawn from the storage unit, $p^\mathrm{st}_{i,j}$, the power traded with the trading partners, $p^\mathrm{tr}_{(i,j),h}$, $j \in \mathcal{N}_i$, and the phase at the
bus $\theta_{i,h}$ over the daily horizon $h = 1\dots,24$. For each $h=1,\dots,24$ and $i \in \mathcal{I}$, we let \begin{align*}
\xbf_{i,h}=\mathrm{col}\left(p^\mathrm{g}_{i,h}, p^\mathrm{mg}_{i,j}, p^\mathrm{st}_{i,j}, \{p^\mathrm{tr}_{(i,j),h}\}_{j \in \mathcal{N}_i}, \theta_{i,h}\right),
\end{align*}
and we denote $\xbf_i:=\mathrm{col}\left(\{\xbf_{i,h}\}_{h=1, \dots, 24}\right)$, $\xbf:=\mathrm{col}\left(\{\xbf_{i}\}_{i \in \mathcal{I}}\right)$. The goal of each agent is to minimize its cost \cite[Eq. (17)]{9732452}:
\begin{equation}
\begin{aligned}
    J_i(\xbf; \ybf)=&\sum_{h=1}^{24} f_{i,h}^\mathrm{g}(p_{i,h}^\mathrm{g})+f_{i,h}^\mathrm{tr}(\{p_{(i,j),h}^\mathrm{tr}\}_{j \in \mathcal{N}_i}\,; \ybf)+ \\
    &f_{i,h}^\mathrm{mg}(p_{i,h}^\mathrm{mg}, p_{-i,h}^\mathrm{mg}),
\end{aligned}
\end{equation}
where $f_{i,h}^\mathrm{g}$ is linear and returns the cost of power generation, $f_{i,h}^\mathrm{mg}$ as in \cite[Eq. (11)]{9732452} returns the cost of buying energy from the main grid, and $f_{i,h}^\mathrm{tr}$ as in \cite[Eq. (7)]{9732452} returns the cost of trading with partner agents and depends on the per-unit price of trading, $\ybf$, set by the community manager. Moreover, each agent $i \in \mathcal{I}$ is subject to local feasible set $\mathcal{X}_i$, which takes into account the satisfaction of the power demand and the operating constraints of the devices, and to affine shared constraints, $g(\xbf)\le \mathbf{0}$, which include the linearized power flow equations as in \cite[Section VI]{benenatix2023optimal}, the trading reciprocity as in \cite[Eq. (8)]{9732452}, and the operating limits of the grid. 

At the upper level, the community manager aims at setting the price of trading of each hour, $\ybf$, in order to minimize the community aggregative cost without deviating too much from a reference price (e.g., agreed through a bilateral contract):
\begin{align*}
    J_0(\ybf, \xbf) = \sum_{i\in \mathcal{I}} J_i(\xbf, \ybf)+\lambda \left\|\ybf-\bar{\ybf}\right\|^2.
\end{align*}

Moreover, the community leader wants to incentivize the agents to minimize the load (hence, the stress) on the grid, maximize the renewable energy production, and reduce the grid imbalances. To do so, the optimal selection criterion that is enforced at the lower-level is:
\begin{equation}\label{phinum}
\begin{aligned}
    \phi(\xbf)=&\sum_{h=1}^{24} \left\|\pbf_h^\mathrm{g}-\bar{\pbf}^\mathrm{g}\right\|^2+\left\|\pbf_h^\mathrm{mg}\right\|^2+\left\|\mathbf{\theta}_h-\bar{\mathbf{\theta}}\right\|^2\\&+\left\|\pbf_h^\mathrm{tr}\right\|^2+\left\|\pbf_h^\mathrm{st}\right\|^2,
\end{aligned}
\end{equation}
where $\bar{\pbf}^\mathrm{g}$ includes the maximum generation production of each agent, while $\bar{\mathbf{\theta}}$ is a vector with all components equal to the phase of the node connected to the main grid \cite{benenatix2023optimal}. We remark that the game satisfies Assumptions \ref{Ass1}, \ref{Ass2}, \ref{ConvexSet}, \ref{Ass4}, and the selection function in \eqref{phinum} satisfies Assumption \ref{StronglyConvex}.
We ran Algorithm \ref{betaZOM} with random initial condition with $\beta_k=\bar{\beta} (k+1)^{-1}$. In Fig. \ref{PLOTGRA}, the convergence of the leader's objective function, under the selection criterion in \eqref{phinum}, is shown. By maintaining a suitable time-scale separation between the updates of the trading price and the incentive term, the community manager can adjust both jointly without hindering convergence. As a result, the objective decreases and eventually converges toward a solution of the Stackelberg game with induced optimal selection.

\begin{figure}
        \includegraphics[width=\columnwidth]{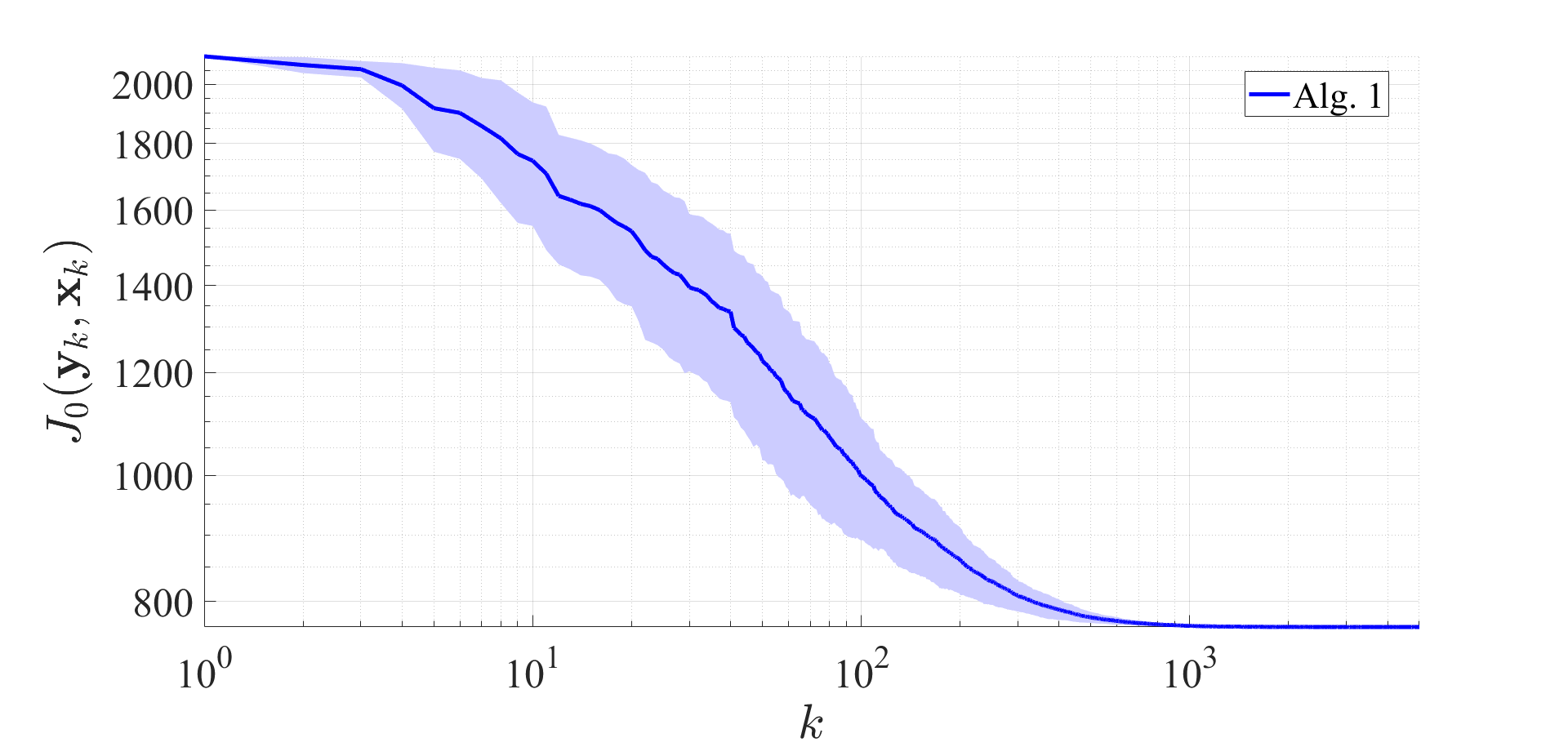}
        \captionsetup{width=\columnwidth}
        \caption{Evolution of $J_0(\ybf_k, \xbf_k)$ throughout the iterations of Alg. \ref{betaZOM} in log-log scale.}
        \label{PLOTGRA}
\end{figure}

\section{Conclusion}\label{Conclusion}
In this paper, we have studied Stackelberg games in which the lower-level admits multiple $v$-GNE, which might potentially lead to ill-posedness. To address this issue, we proposed a framework that guarantees well-posedness by enforcing optimal equilibrium selection at the lower level through a leader-designed incentive term embedded in the followers’ game. Next, we developed a follower-agnostic zeroth-order method that enables the leader to iteratively update both its decision variable and the incentive, without requiring prior knowledge of the followers’ model. We showed that, by maintaining a suitable time-scale separation between the updates of these two terms, convergence to a solution of the resulting Stackelberg game with induced optimal selection is guaranteed. Finally, we validated the proposed approach in the context of energy communities, where agents engage in peer-to-peer trading under prices set by a community manager acting as the leader. Our numerical results show that the additional incentive term enables the community manager to steer the selection of energy profiles towards a desired equilibrium.

\bibliographystyle{IEEEtran}
\bibliography{biblio}
\appendix

\emph{Proof of Theorem \ref{BoundZOM}:}
The gradient estimator $\hat{\gbf}_k$ in \eqref{Gradient} is affected by three error components: $\mathcal{E}_1:=\mathbb{E}\left[\gbf_k\right]-\nabla J_0(\ybf_k,\xbf^*_\phi(\ybf_k))$ represents the error due to using an approximation of the true gradient, $\mathcal{E}_2:=\gbf_k-\mathbb{E}\left[\gbf_k\right]$ is the bias due to the randomness of the gradient estimator $\gbf_k$, and $\mathcal{E}_3:=\hat{\gbf}_k-\gbf_k$ denotes the additional bias introduced by our proposed estimator $\hat{\gbf}_k$. By separately bounding each error term, we show that $J_0(\ybf_k, x^*_\phi(\ybf_k))$ approximately decreases throughout the iterates $\left(\ybf_k\right)_{k \in [K]}$.  

While the presence of $\mathcal{E}_1$ and $\mathcal{E}_2$ is standard in the zeroth-order method literature \cite[Theorem 1]{maheshwari2024follower}, $\mathcal{E}_3$ is the novel term of our approach, for which time-scale separation between $\left(\eta_k\right)_{k \in [K]}$ and $\left(\beta_k\right)_{k \in [K]}$ must be ensured. 

The error term $\mathcal{E}_3$ can be bounded as follows:
\begin{multline}
    \left\|\gbf_k-\hat{\gbf}_k\right\|,^2 \le \frac{m^2}{\delta_k^2} \left\|J_0(\ybf_k, x_{\beta_k}(\ybf_k))-J_0(\ybf_k, x_{\phi}^*(\ybf_k))\right\|^2 \\
    + \frac{m^2}{\delta_k^2}\left\|J_0(\hat{\ybf}_k, x_{\beta_k}(\hat{\ybf}_k))-J_0(\hat{\ybf}_k, x_{\phi}^*(\hat{\ybf}_k))\right\|^2 \le \\
    \frac{m^2L_2^2}{\delta_k^2} \left(\left\|x_{\beta_k}(\ybf_k)-x_\phi^*(\ybf_k)\right\|^2+\left\|x_{\beta_k}(\hat{\ybf}_k)-x_\phi^*(\hat{\ybf}_k)\right\|^2\right) \le \\ 
    \frac{m^2L_2^2\beta_k^2}{\delta_k^2 \mu^2} \left(\left\|\nabla\phi(x^*_\phi(\ybf_k))\right\|^2+\left\|\nabla\phi(x^*_\phi(\hat{\ybf}_k))\right\|^2\right),
\end{multline}
where we exploited the Lipschitzness of $J_0(\ybf, \xbf)$ and the fact that $\left\|x_\beta(\ybf_k)-x_\phi^*(\ybf_k)\right\|\le \frac{\beta}{\mu}\left\|\nabla \phi(x_\phi^*(\ybf_k))\right\|$ \cite[Theorem 12.2.5]{facchinei2003finite}.

Also, by upper bounding $\left\|\nabla \phi(x_\phi^*(\ybf_k))\right\|$, we get 
\begin{align}\label{E3}
\mathcal{E}_3^2 \le \frac{m^2L_2^2\beta_k^22B^2}{\delta_k^2 \mu^2}.
\end{align}
Next, the proof follows similar steps as \cite[Proof of Theorem 1]{maheshwari2024follower}, with the only difference that $\mathcal{E}_3$ is different, and upper bounded as in \eqref{E3}. As such, from \cite[Lemma 1, Lemma 2]{maheshwari2024follower}, we obtain:
\begin{multline*}
        \frac{1}{\sum_{k=1}^K \eta_k} \sum_{k=1}^K \eta_k \mathbb{E}\left[\left\|\nabla J_0(\ybf_k,\xbf^*_\phi(\ybf_k))\right\|^2\right]\le \\\mathcal{U}_1 + \mathcal{U}_2 + \frac{m^2L_2^2 4 B^2}{\mu^2 } \frac{\sum_{k=1}^K\beta_k^2\eta_k\delta_k^{-2}}{\sum_{k=1}^K\eta_k}+\mathcal{U}_4.
        \end{multline*}
By choosing $\beta_k=\bar{\beta}(k+1)^{-\alpha}$ with $\alpha > \frac{1}{2}$, the bound in \eqref{BoundZOM} is recovered, which concludes the proof.
\end{document}